\documentclass[12pt,a4paper]{article}
\usepackage{amssymb} 
\usepackage{amsmath} 
\usepackage{amsthm} 
\usepackage{hyperref}
\usepackage{a4wide}

\newtheorem{theorem}{Theorem}
\newtheorem{problem}[theorem]{Problem}
\newtheorem{proposition}[theorem]{Proposition}
\newtheorem{lemma}[theorem]{Lemma}
\newtheorem{corollary}[theorem]{Corollary}
\newtheorem{conjecture}[theorem]{Conjecture}
\newtheorem*{chernoff}{A Chernoff Bound}
\newtheorem*{slll}{The Lov\'asz Local Lemma}
\newtheorem*{claim}{Claim}

\newcommand{\Exp}{\,\mathbb{E}}
\renewcommand{\Pr}{\,\mathbb{P}}
\newcommand{\eps}{\varepsilon}

\DeclareMathOperator{\Bin}{Bin}

\DeclareMathOperator{\ch}{ch}
\DeclareMathOperator{\F}{\mathcal F}

\newcommand*{\myproofname}{Proof}
\newenvironment{claimproof}[1][\myproofname]{\begin{proof}[#1]}{\end{proof}}

\title{Asymmetric list sizes in bipartite graphs}

\author{
Noga Alon
\thanks{Department of Mathematics,
Princeton University, Princeton, NJ 08544, USA
and
Schools of Mathematics and Computer Science, Tel Aviv University,
Tel Aviv 6997801, Israel.
Email: \protect\href{mailto:nogaa@tau.ac.il}{\protect\nolinkurl{nogaa@tau.ac.il}}. Supported in part by NSF grant DMS-1855464, BSF grant 2018267 and the Simons Foundation.}
\and
Stijn Cambie
\thanks{Department of Mathematics, Radboud University, Postbus 9010, 6500 GL Nijmegen, Netherlands. 
Email: \protect\href{mailto:stijn.cambie@hotmail.com}{\protect\nolinkurl{stijn.cambie@hotmail.com}}, \protect\href{mailto:ross.kang@gmail.com}{\protect\nolinkurl{ross.kang@gmail.com}}. Supported by a Vidi grant (639.032.614) of the Netherlands Organisation for Scientific Research (NWO).}
\and
Ross J. Kang
\footnotemark[2]
}

\begin{document}

\maketitle

\begin{abstract}
Given a bipartite graph with parts $A$ and $B$ having maximum degrees at most $\Delta_A$ and $\Delta_B$, respectively, consider a list assignment such that every vertex in $A$ or $B$ is given a list of colours of size $k_A$ or $k_B$, respectively.

We prove some general sufficient conditions in terms of $\Delta_A$, $\Delta_B$, $k_A$, $k_B$ to be guaranteed a proper colouring such that each vertex is coloured using only a colour from its list. 
These are asymptotically nearly sharp in the very asymmetric cases.
We establish one sufficient condition in particular, where $\Delta_A=\Delta_B=\Delta$, $k_A=\log \Delta$ and $k_B=(1+o(1))\Delta/\log\Delta$ as $\Delta\to\infty$.
This amounts to partial progress towards a conjecture from 1998 of Krivelevich and the first author.

We also derive some necessary conditions through an intriguing connection between the complete case and the extremal size of approximate Steiner systems. 
We show that for complete bipartite graphs these conditions are asymptotically nearly sharp in a large part of the parameter space. This has provoked the following.

In the setup above, we conjecture that a proper list colouring is always guaranteed
\begin{itemize}
\item if $k_A \ge \Delta_A^\varepsilon$ and $k_B \ge \Delta_B^\varepsilon$ for any $\varepsilon>0$ provided $\Delta_A$ and $\Delta_B$ are large enough;
\item if $k_A \ge C \log\Delta_B$ and $k_B \ge C \log\Delta_A$ for some absolute constant $C>1$; or
\item if $\Delta_A=\Delta_B = \Delta$ and $ k_B \ge C (\Delta/\log\Delta)^{1/k_A}\log \Delta$ for some absolute constant $C>0$.
\end{itemize}
These are asymmetric generalisations of the above-mentioned conjecture of Krivelevich and the first author, and if true are close to best possible.
Our general sufficient conditions provide partial progress towards these conjectures.
\end{abstract}


\section{Introduction}
\label{sec:intro}

List colouring of graphs, whereby arbitrary restrictions on the possible colours used per vertex are imposed, was introduced independently by Erd\H{o}s, Rubin and Taylor~\cite{ERT80} and by Vizing~\cite{Viz76}.
Let $G = (V,E)$ be a simple, undirected graph. 
For a positive integer $k$, a mapping $L: V\to \binom{{\mathbb Z}^+}{k}$ is called a {\em $k$-list-assignment} of $G$; a colouring $c: V\to {\mathbb Z}^+$ is called an {\em $L$-colouring} if $c(v)\in L(v)$ for any $v\in V$. We say $G$ is {\em $k$-choosable} if for any $k$-list-assignment $L$ of $G$ there is a proper $L$-colouring of $G$. The {\em choosability $\ch(G)$} (or {\em choice number} or {\em list chromatic number}) of $G$ is the least $k$ such that $G$ is $k$-choosable.

Note that if $G$ is $k$-choosable then it is properly $k$-colourable.
On the other hand, Erd\H{o}s, Rubin and Taylor~\cite{ERT80} exhibited bipartite graphs with arbitrarily large choosability. 
Let $K_{n_1,n_2}$ denote a complete bipartite graph with part sizes $n_1$ and $n_2$.
\begin{theorem}[\cite{ERT80}]\label{thm:ERT}
For some function $M(k)$ with $M(k) = 2^{k+o(k)}$ as $k\to\infty$,
\begin{enumerate}
\item\label{itm:i}
$K_{n,n}$ is not $k$-choosable if $n\ge M(k)$, and
\item
$K_{n_1,n_2}$ is $k$-choosable if $n_1+n_2 < M(k)$.
\end{enumerate}
\end{theorem}

\noindent
More succinctly Theorem~\ref{thm:ERT} says that $\ch(K_{n,n}) \sim \log_2 n$ as $n\to\infty$.

There are two contrasting ways to try to strengthen this last statement.
First does the lower bound hold more generally; that is, does a bipartite graph with minimum degree $\delta$ have choosability $\Omega(\log \delta)$ as $\delta\to\infty$? Indeed this was shown by the first author~\cite{Alo00} using a probabilistic argument; later Saxton and Thomason~\cite{SaTh15} proved the asymptotically optimal lower bound of $(1+o(1))\log_2 \delta$.
Second does the upper bound hold more generally; that is, does a bipartite graph with maximum degree $\Delta$ always have choosability $O(\log \Delta)$ as $\Delta\to\infty$? Krivelevich and the first author conjectured this in 1998.

\begin{conjecture}[\cite{AlKr98}]\label{conj:AlKr}
There is some absolute constant $C>0$ such that
any bipartite graph of maximum degree at most $\Delta$
is $k$-choosable 
if $k \ge C\log \Delta$.
\end{conjecture}

\noindent
Johansson's result for triangle-free graphs~\cite{Joh96} gives the conclusion with $k\ge C\Delta/\log\Delta$, which is far from the conjectured bound.  
Conjecture~\ref{conj:AlKr} is an elegant problem in a well-studied area, but apart from constant factors\footnote{In a recent advance, Molloy~\cite{Mol19} considerably improved on the constant in Johansson's, cf.~also~\cite{DJKP21}.}  
there has been no progress until now.

Despite the apparent difficulty of Conjecture~\ref{conj:AlKr}, we propose an asymmetric refinement.
Given  a bipartite graph $G=(V{=}A\cup B,E)$ with parts $A$, $B$ and positive integers $k_A$, $k_B$, a mapping $L: A\to \binom{{\mathbb Z}^+}{k_A},B\to \binom{{\mathbb Z}^+}{k_B}$ is called a {\em $(k_A,k_B)$-list-assignment} of $G$. We say $G$ is {\em $(k_A,k_B)$-choosable} if there is guaranteed a proper $L$-colouring of $G$ for any such $L$.

\begin{problem}\label{prob:asymmetric}
Given $\Delta_A$ and $\Delta_B$, what are optimal choices of $k_A\le \Delta_A$ and $k_B \le \Delta_B$ for which 
any bipartite graph $G=(V{=}A\cup B,E)$ with parts $A$ and $B$ having maximum degrees at most $\Delta_A$ and $\Delta_B$, respectively, is $(k_A,k_B)$-choosable?
\end{problem}
\noindent 
We have the upper bounds on $k_A$, $k_B$, since the problem is trivial if $k_A>\Delta_A$ or $k_B>\Delta_B$.

Note that Problem~\ref{prob:asymmetric}, since it has a higher-dimensional parameter space, has wider scope than Conjecture~\ref{conj:AlKr} and is necessarily more difficult.
However, the extra generality in Problem~\ref{prob:asymmetric} has permitted a glimpse at an unexpected and basic connection between list colouring and combinatorial design theory (Theorem~\ref{thm:completesteiner}).
This has prompted us to explore specific areas of the parameter space in Problem~\ref{prob:asymmetric}, motivating concrete versions of Problem~\ref{prob:asymmetric} in the spirit of Conjecture~\ref{conj:AlKr} (Conjecture~\ref{conj:general}).
One hope of ours is that further study of these problems may yield insights into 
Conjecture~\ref{conj:AlKr}.
But in fact, already in the present work, we have obtained asymmetric progress towards Conjecture~\ref{conj:AlKr} (Corollary~\ref{cor:main}).

Our first main result provides general progress towards Problem~\ref{prob:asymmetric}.

\begin{theorem}\label{thm:general}
Let the positive integers $\Delta_A$, $\Delta_B$, $k_A$, $k_B$, with $k_A\le \Delta_A$ and $k_B\le \Delta_B$, satisfy one of the following conditions.
\begin{enumerate}
\item\label{itm:transversals}
$k_B \ge (ek_A\Delta_B)^{1/k_A}\Delta_A$.
\item\label{itm:coupon}
\(\displaystyle
e(\Delta_A(\Delta_B-1)+1) \left(1-(1-1/k_B)^{\Delta_A\min\left\{1,k_B/k_A\right\}}\right)^{k_A} \le 1.
\)
\end{enumerate}
Then any bipartite graph $G=(V{=}A\cup B,E)$ with parts $A$ and $B$ having maximum degrees at most $\Delta_A$  and $\Delta_B$, respectively, is $(k_A,k_B)$-choosable.
\end{theorem}

\noindent
Theorem~\ref{thm:general} under condition~\ref{itm:transversals} follows from a simple application of the Lov\'asz Local Lemma, as we show in Section~\ref{sec:transversals}. This sufficient condition for Problem~\ref{prob:asymmetric} is related to independent transversals in hypergraphs, cf.~\cite{EGL94,Hax01}, and to single-conflict chromatic number~\cite{DEKO21}.
In the most asymmetric settings (when $k_A$ and $\Delta_A$ are fixed constants), condition~\ref{itm:transversals} is sharp up to a constant factor.
We prove Theorem~\ref{thm:general} under condition~\ref{itm:coupon} in Section~\ref{sec:coupon} with the Lov\'asz Local Lemma and a link to the coupon collector problem. 
In an attempt to clear up the `parameter soup' arising from Problem~\ref{prob:asymmetric} and the above sufficient conditions, we will discuss specific natural context for these after presenting some further results.

We complement our sufficient conditions for $(k_A,k_B)$-choosability with necessary ones, given mainly by the complete bipartite graphs.
An easy boundary case is provided as follows, for warmup. This generalises a classic non-$k$-choosable construction.
\begin{proposition}\label{prop:classic}
For any $\delta,k\ge 2$, the complete bipartite graph $G=(V{=}A\cup B,E)$
with $|A|=\delta^k$ and $|B|=k$ is not $(k,\delta)$-choosable.
\end{proposition}
\begin{proof}
Let the vertices of $B$ be assigned $k$ disjoint lists of length $\delta$, and let the vertices of $A$ be assigned all possible $k$-tuples drawn from these $k$ disjoint lists. 
\end{proof}

\noindent
This is best possible in the sense that the conclusion does not hold if $|A|<\delta^k$ or $|B|<k$; however, in Section~\ref{sec:improvement} we exhibit a non-complete  non-$(k,\delta)$-choosable construction that is a slightly more efficient (in the sense that it has $\Delta_A=k$ and $\Delta_B< \delta^k$).
On the other hand, Proposition~\ref{prop:classic} shows that condition~\ref{itm:transversals} in  Theorem~\ref{thm:general} cannot be relaxed much in general. This follows for instance by considering the most asymmetric case, namely $k=k_A=\Delta_A$ and $\delta=\Delta_B^{1/k}$, with $k$ fixed and $\Delta_B \to\infty$.
Thus for fixed $k_A$ and $\Delta_A$ Problem~\ref{prob:asymmetric} is settled up to a constant factor.

Our second main result is a related but much broader necessary condition for Problem~\ref{prob:asymmetric}, also via the complete case.
Let us write  $\overline{M}(k_1,k_2,\ell)$ for the hypergraph Tur\'an number defined as
 the minimum number of edges in a $k_2$-uniform hypergraph on $\ell$ vertices with no independent set of size $\ell-k_1.$
 This parameter is equivalent to the extremal size of approximate Steiner systems;
in particular, $\overline{M}(k_1,k_2,\ell)$ is equal to the cardinality of a smallest $k_1$-$(\ell,\ell-k_2,\mathbb{Z}^+)$ design, cf.~\cite[Ch.~13]{ES74} and~\cite{Kee11}.
(Recall that a $t$-$(v,k,\mathbb{Z}^+)$ design is a collection of $k$-element subsets, called blocks, of some $v$-element set $X$ such that each $t$-element subset of $X$ is contained in at least one block.)
We can draw a link between this classical extremal parameter and Problem~\ref{prob:asymmetric}.
We show the following result in Section~\ref{sec:completesteiner}. 

\begin{theorem}\label{thm:completesteiner}
	Let $k_A,k_B,\ell, k_1,k_2$ be integers such that $k_A,k_B\ge 2$ and $\ell = k_1+k_2+1$.
	The complete bipartite graph $G=(V{=}A\cup B,E)$ with $|A|=\overline{M}(k_1,k_A,\ell)$ and $|B|=\overline{M}(k_2,k_B,\ell)$ is not $(k_A,k_B)$-choosable.
\end{theorem}

\noindent
This link allows us, using known results for $\overline{M}$, to read off decent necessary conditions for specific parameterisations of Problem~\ref{prob:asymmetric}. As one example, a lower bound on $\ch(K_{n,n})$ of the form $\ch(K_{n,n})\gtrsim \frac12\log_2 n$ follows from Theorem~\ref{thm:completesteiner} with the choice $k_1=k_2=k-1$, $k_A=k_B=k$, and $\ell=2k-1$ for some $k$.
As another example, a slightly weaker form of the necessary condition of Proposition~\ref{prop:classic} follows with the choice $k_1=k_A(k_B-1)$, $k_2=k_A-1$, and $\ell=k_Ak_B$.
We detail both of these easy examples in Section~\ref{sec:completesteiner}.

Note though that these last two examples show that Theorem~\ref{thm:completesteiner} provides suboptimal necessary conditions for $(k_A,k_B)$-choosability even in the complete case. Furthermore, in Section~\ref{sec:improvement} we give a construction to show that the complete case cannot, in general, be precisely extremal for Problem~\ref{prob:asymmetric}.
Nevertheless,  we surmise that Theorem~\ref{thm:completesteiner} provides some good rough borders for Problem~\ref{prob:asymmetric}.
More specifically, in Section~\ref{sec:SufComplete} we give some basic sufficient conditions for  $(k_A,k_B)$-choosability specific to the complete case and show in Section~\ref{sec:asymptoticcomplete}, through some routine asymptotic calculus, how these roughly match with Theorem~\ref{thm:completesteiner} over a broad family of parameterisations for Problem~\ref{prob:asymmetric}. 
Then, just as Conjecture~\ref{conj:AlKr} was informed by Theorem~\ref{thm:ERT}, the asymptotic behaviour of $(k_A,k_B)$-choosability in the complete case leads us to conjecture the following concrete versions of Problem~\ref{prob:asymmetric}.

\begin{conjecture}\label{conj:general}
Let the positive integers $\Delta_A$, $\Delta_B$, $k_A$, $k_B$ satisfy one of the following.
\begin{enumerate}
\item\label{itm:uncrossed}
Given $\eps>0$, we have $\Delta_A,\Delta_B\ge \Delta_0$ for some $\Delta=\Delta_0(\eps)$, and 
\begin{align*}
k_A \ge \Delta_A^\eps
\quad\text{and}\quad
k_B \ge \Delta_B^\eps.
\end{align*}
\item\label{itm:crossed}
For some absolute constant $C>1$,
\begin{align*}
k_A \ge C \log\Delta_B
\quad\text{and}\quad
k_B \ge C \log\Delta_A.
\end{align*}
\item\label{itm:symmetricasymmetric}
$\Delta_A=\Delta_B=\Delta$, and,
for some absolute constant $C>0$,
\begin{align*}
k_B \ge C (\Delta/\log\Delta)^{1/k_A}\log \Delta
\quad\text{or}\quad
k_A \ge C (\Delta/\log\Delta)^{1/k_B}\log \Delta.
\end{align*}
\end{enumerate}
Then any bipartite graph $G=(V{=}A\cup B,E)$ with parts $A$ and $B$ having maximum degrees at most $\Delta_A$  and $\Delta_B$, respectively, is $(k_A,k_B)$-choosable.
\end{conjecture}

\noindent
We note that in Section~\ref{sec:SufComplete} we also show how this conjecture holds for the complete bipartite graphs (Theorem~\ref{thr:3CasesChoos}).
Conjecture~\ref{conj:general} constitutes three natural asymmetric analogues of Conjecture~\ref{conj:AlKr}. The first is weaker than Conjecture~\ref{conj:AlKr}; the latter two are stronger.
We discuss these separately in turn.

Behind condition~\ref{itm:transversals}, Conjecture~\ref{conj:general} concerns the question, for what positive functions $f$ does $k_A\ge f(\Delta_A)$ and $k_B\ge f(\Delta_B)$ suffice for $(k_A,k_B)$-choosability, no matter how far apart $\Delta_A$ and $\Delta_B$ are? Essentially we posit that some $f$ with $f(x)=x^{o(1)}$ as $x\to\infty$ will work, and this would be best possible due to the complete case (Theorem~\ref{thm:completecomparison}\ref{itm:uncrossedcomplete,nec}). In particular, $f(x)=O(\log x)$ is impossible, in contrast to Conjecture~\ref{conj:AlKr}.
But in fact, Theorem~\ref{thm:general} implies that Conjecture~\ref{conj:general} under condition~\ref{itm:transversals} reduces to its most symmetric form.

\begin{corollary}
Given $0<\eps<1$
	and integers $\Delta_A$, $\Delta_B$ satisfying $\Delta_B> \Delta_A^{2/\eps}$ and $\Delta_A> (4/ \eps)^{1/ \eps}$,
 any bipartite graph $G=(V{=}A\cup B,E)$ with parts $A$ and $B$ having maximum degrees at most $\Delta_A$  and $\Delta_B$, respectively, is $(\Delta_A^{\eps}, \Delta_B^{\eps})$-choosable.

	As a consequence, Conjecture~\ref{conj:general} under condition~\ref{itm:uncrossed} follows from the same assertion under the further assumption that $\Delta_A=\Delta_B=\Delta$.
\end{corollary}

\begin{proof}
	It suffices to check condition~\ref{itm:transversals} of Theorem~\ref{thm:general} with $k_A=\Delta_A^{\eps}$ and $k_B=\Delta_B^{\eps}$. Indeed, as $\Delta_B^2 > e \Delta_A^{\eps} \Delta_B$, $\eps/4>1/\Delta_A ^{\eps}$ and $\Delta_B^{\eps/2}> \Delta_A$, we have
	$\Delta_B^{\eps}> \left( e \Delta_A^{\eps} \Delta_B \right) ^{1/\Delta_A^{\eps}}\Delta_A$.
	

Assume condition~\ref{itm:uncrossed} of Conjecture~\ref{conj:general}, and moreover assume the truth of the conjecture only for its most symmetric form $\Delta_A=\Delta_B=\Delta$. We may assume that $\Delta_0>(4/ \eps)^{1/ \eps}$, and so, without loss of generality, we may also assume by the first part that $\Delta_A \le \Delta_B \le  \Delta_A^{2/\eps}$. For any bipartite graph $G=(V{=}A\cup B,E)$ with parts $A$ and $B$ having maximum degrees at most $\Delta_A$  and $\Delta_B$, we have from the assumption that $G$ is $k$-choosable for, say, $k=\Delta_B^{\eps^3}\le \Delta_A^{2\eps^2}$, provided $\Delta_B$, and thus $\Delta_A$, is large enough as a function of $\eps$. This is smaller than $\Delta_A^{\eps}$ for $\Delta_A$ sufficiently large as a function of $\eps$, as required.
\end{proof}

\noindent
Thus Conjecture~\ref{conj:general} under condition~\ref{itm:uncrossed} would follow from a weaker form of Conjecture~\ref{conj:AlKr}.
By an earlier work due to Davies, de Joannis de Verclos, Pirot and the third author~\cite{DJKP21}, we so far know that $f(x)=(1+o(1))x/\log x$ works.

Under condition~\ref{itm:crossed}, Conjecture~\ref{conj:general} concerns a `crossed' version of the previous question, so for what positive functions $g$ does $k_A \ge g(\Delta_B)$ and $k_B \ge g(\Delta_A)$ suffice for $(k_A,k_B)$-choosability? In this case, we conjecture that $g(x)=O(\log x)$ will work, which coincides with Conjecture~\ref{conj:AlKr} in the symmetric case $\Delta_A=\Delta_B=\Delta$.
The complete bipartite graphs demonstrate the hypothetical sharpness of this assertion up to a constant factor for nearly the entire range of possibilities for $\Delta_A$ and $\Delta_B$ (Theorem~\ref{thm:completecomparison}\ref{itm:crossedcomplete,nec}).
Some modest partial progress towards Conjecture~\ref{conj:general} under condition~\ref{itm:crossed} follows from Theorem~\ref{thm:general}.

\begin{corollary}
Given $\eps>0$, there exists $\delta$ such that for $\Delta_B$ large enough, 
any bipartite graph $G=(V{=}A\cup B,E)$ with parts $A$ and $B$ having maximum degrees at most $\Delta_A$  and $\Delta_B$, respectively, is $(\delta\log \Delta_B, (1+\eps)\Delta_A)$- and $(\log \Delta_B, (e+\eps)\Delta_A)$-choosable.
\end{corollary}

\begin{proof}
	This follows for $\Delta_B$ large enough from condition~\ref{itm:transversals} of Theorem~\ref{thm:general} with either $k_A=\delta\log\Delta_B$ and $k_B= (1+\eps)\Delta_A$ or with $k_A=\log \Delta_B$ and $k_B=(e+\eps)\Delta_A$.
\end{proof}

Under condition~\ref{itm:symmetricasymmetric}, Conjecture~\ref{conj:general} concerns the setting most closely related to Conjecture~\ref{conj:AlKr}. It suggests how Problem~\ref{prob:asymmetric} might behave for the symmetric case $\Delta_A=\Delta_B=\Delta$.
The complete bipartite graphs demonstrate its hypothetical sharpness up to a constant factor for the entire range of possibilities for $k_A$, and thus symmetrically $k_B$ (Theorem~\ref{thm:completecomparison}\ref{itm:symmetricasymmetriccomplete,nec}).
Theorem~\ref{thm:general} provides the following partial progress towards Conjecture~\ref{conj:general} under condition~\ref{itm:symmetricasymmetric}. In fact, this constitutes significant (asymmetric) progress towards Conjecture~\ref{conj:AlKr}, and is a first concrete step in this longstanding problem.

\begin{corollary}\label{cor:main}
	Given $\eps>0$, any bipartite graph $G=(A \cup B, E)$ with parts $A$ and $B$ having maximum degree at most $\Delta$ is 
	$((1+\eps) \Delta/\log_4 \Delta, 2 )$- and $((1+\eps) \Delta/\log \Delta, \log \Delta  )$-choosable
	for all $\Delta$ large enough.
\end{corollary}

\begin{proof}
	This follows for $\Delta$ large enough from condition~\ref{itm:coupon} in Theorem~\ref{thm:general} with $\Delta_A=\Delta_B =\Delta$ and either $k_A= (1+\eps) \Delta/\log_4 \Delta$ and $k_B = 2$ or $k_A= (1+\eps) \Delta/\log \Delta$ and $k_B = \log \Delta$.
\end{proof}

In Section~\ref{sec:boundary} we are able to nearly completely characterise $(k_A,k_B)$-choosability for complete bipartite graphs if $k_A\ge \Delta_A-1$.
In Section~\ref{sec:degrees}, we give a connection between minimum degree and $(k_A,k_B)$-choosability along the lines of~\cite{Alo00}.

\subsection*{Probabilistic preliminaries}

We will use the following standard probabilistic tools, cf.~\cite{AlSp16}.

\begin{chernoff}
For $X\sim \Bin(n,p)$ and $\eps \in [0,1]$, $\Pr(X < (1-\eps)np) \le \exp(-\frac{\eps^2}{2}np)$.
\end{chernoff}

\begin{slll}
Consider a set $\cal E$ of (bad) events such that for each $A\in \cal E$
\begin{enumerate}
\item $\Pr(A) \le p < 1$, and
\item $A$ is mutually independent of a set of all but at most $d$ of the other events.
\end{enumerate}
If $ep(d+1)\le1$, then with positive probability none of the events in $\cal E$ occur.
\end{slll}


\section{A sufficient condition via transversals}\label{sec:transversals}

In this section, we prove Theorem~\ref{thm:general} under condition~\ref{itm:transversals}.
We prefer to state and prove a slightly stronger form. To do so, we need some notational setup.
Let $H=(V,E)$ be a hypergraph. The {\em degree} of a vertex in $H$ is the number of edges containing it. 
Given some partition of $V$, a {\em transversal} of $H$ is a subset of $V$ that intersects each part in exactly one vertex. A transversal of $H$ is called {\em independent} if it contains no edge, cf.~\cite{EGL94}.

\begin{lemma}\label{thm:transversal} Fix $k\ge 2$.
Let $H$ be a $k$-uniform vertex-partitioned hypergraph, each part being of size $\ell$, such that every part has degree sum at most $\Delta$.  If $\ell^k \ge e(k(\Delta-1)+1)$, then $H$ has an independent transversal.
\end{lemma}

\noindent
Let us first show that Lemma~\ref{thm:transversal} implies Theorem~\ref{thm:general} under condition~\ref{itm:transversals}.

\begin{proof}[Proof of Theorem~\ref{thm:general} under condition~\ref{itm:transversals}]
Let $k_A,k_B,\Delta_A,\Delta_B$ satisfy condition~\ref{itm:transversals}. Let $L$ be a $(k_A,k_B)$-list-assignment of $G$. We would like to show that there is a proper $L$-colouring of $G$. We do so by defining a suitable hypergraph $H=(V_H,E_H)$.

Let $(w,c)$ be a vertex of $V_H$ if $w\in B$ and $c\in L(w)$.

Let $((w_1,c_1),\dots,(w_{k_A},c_{k_A}))$ be an edge of $E_H$ whenever there is some $v\in A$ such that $N(v) \supseteq \{w_1,\dots,w_{k_A}\}$ and $L(v) = \{c_1,\dots,c_{k_A}\}$.

Note that $H$ is a $k_A$-uniform vertex-partitioned hypergraph, where the parts are naturally induced by each list in $B$ and so are each of size $k_B$. We have defined $H$ and its partition so that any independent transversal corresponds to a special partial $L$-colouring of $G$. In particular, it is an $L$-colouring of the vertices in $B$ for which there is guaranteed to be a colour in $L(v)$ available for every $v\in A$, and so it can be extended to a proper $L$-colouring of all $G$.

Every part in $H$ has degree sum at most
\(
 \Delta_B \binom{ \Delta_A -1}{ k_A -1} k_A! \le 
 \Delta_B \Delta_A^{k_A},
\)
so the result follows from Lemma~\ref{thm:transversal} with $\ell=k_B$ and $\Delta=\Delta_B \Delta_A^{k_A}$.
\end{proof}

\begin{proof}[Proof of Lemma~\ref{thm:transversal}]
Write $H=(V,E)$ and suppose  
$\ell^k \ge e(k(\Delta-1)+1)$. 
Consider the random transversal $\mathbf{T}$ formed by choosing one vertex from each part independently and uniformly. For each edge $f\in E$, let $A_f$ denote the event that $\mathbf{T}\supseteq f$. Note that $\Pr(A_f) \le 1/\ell^k$. Moreover $A_f$ is mutually independent of a set of all but at most 
$k(\Delta-1)$
of the other events $A_{f'}$. The transversal $\mathbf{T}$ is independent if none of the events $A_f$ occur. Since by assumption $e (1/\ell^k)(k(\Delta-1)+1) \le 1$, there is a positive probability that $\mathbf{T}$ is independent by the Lov\'asz Local Lemma.
\end{proof}

\section{A sufficient condition via coupon collection}\label{sec:coupon}

\begin{proof}[Proof of Theorem~\ref{thm:general} under condition~\ref{itm:coupon}]
Let $k_A,k_B,\Delta_A,\Delta_B$ satisfy condition~\ref{itm:coupon}. Let $L$ be a $(k_A,k_B)$-list-assignment of $G$. We would like to show that there is a proper $L$-colouring of $G$.
	To this end, colour each vertex $w\in B$, randomly and independently, by a colour chosen uniformly from its list $L(w)$.
	Let $T_{v,c}$ be the event that $v \in A$ has a neighbour coloured with colour $c$.
	Let $T_v$ be the event that $T_{v,c}$ happens for all $c\in L(v)$.
	The proof hinges on the following claim, which is related to the coupon collector problem, cf.~e.g.~\cite{Doe20}.

\begin{claim}
	The events $T_{v,c}$, for fixed $v$ as $c$ ranges over all colours in $L(v)$, are negatively correlated. 
	In particular,
	\(\Pr(T_v) \le \prod_{c\in L(v)} \Pr(T_{v,c})\).
\end{claim}

\begin{claimproof}
	We have to prove, for every $I \subset L(v)$, that
	$\Pr(\forall c \in I \colon T_{v,c}) \le \prod_{c\in I} \Pr(T_{v,c}).$
	If some colour $c \in L(v)$ is not in the list of any neighbour of $v$, both sides equal zero and so the inequality holds. So assume this is not the case.
	We prove the statement by induction on $\lvert I \rvert.$ When $\lvert I \rvert\le 1$ the statement is trivially true.
	Let $I \subset L(v)$ be a subset for which the statement is true and let $c' \in L(v) \setminus I.$ We now prove the statement for $I'=I \cup \{c'\}.$
	We have $\Pr(\forall c \in I \colon T_{v,c}) \le \Pr(\forall c \in I \colon T_{v,c} \mid \lnot T_{v,c'} )$ as the probability to use a colour in $I$ is larger if in all neighbouring lists the colour $c'$ is removed.
	This is equivalent to 
	\begin{align*}
	\Pr(\forall c \in I \colon T_{v,c}) &\ge \Pr(\forall c \in I \colon T_{v,c} \mid T_{v,c'} )\\
	\iff\Pr(\forall c \in I' \colon T_{v,c}) &\le 	\Pr(\forall c \in I \colon T_{v,c}) \Pr(T_{v,c'})
	\end{align*}
	This last expression is at most $ \prod_{c\in I'} \Pr(T_{v,c})$ by the induction hypothesis, as desired.
\end{claimproof}

For the $i^{\text{th}}$ colour $c_i$ in $L(v)$, let the number of occurrences of $c_i$ in the neighbouring lists of $v$ be $x_i$.
	Note that $\Pr(T_{v,c_i}) = 1- (1-1/k_B)^{x_i}.$
	Using  $x_i \le \Delta_A$ for every $1 \le i \le k_A$ and the claim, 
	 we have 
	$$\Pr(T_v) \le \left(1- \left(1-\frac{1}{k_B}\right)^{  \Delta_A} \right)^{k_A}.$$
	Noting that $\sum_{i=1}^{k_A} x_i \le k_B \Delta_A$ and that the function $\log (1- (1-1/k_B)^{x} )$ is concave and increasing, Jensen's Inequality applied with the claim  
	implies that 
	$$\Pr(T_v) \le \left(1- \left(1-\frac{1}{k_B}\right)^{ k_B \Delta_A/k_A} \right)^{k_A}.$$
	Each event $T_v$ is mutually independent of all other events $T_u$ apart from those corresponding to vertices $u\in A$ that have a common neighbour with $v$ in $G$. As there are at most $\Delta_A(\Delta_B-1)$ such vertices besides $v$, the Lov\'asz Local Lemma guarantees with positive probability that none of the events $T_v$ occur, i.e.~there is a proper $L$-colouring, as desired.
\end{proof}

\section{The complete case and Steiner systems}\label{sec:completesteiner}

In this section, we investigate general necessary conditions for $(k_A,k_B)$-choosability via the complete bipartite graphs.
Inspired in part by Theorem~\ref{thm:ERT} and related work of Bonamy and the third author~\cite{BoKa17}, this leads naturally to the study of an extremal set theoretic parameter.
For positive integers $k_1,k_2,\ell$, we say that a family ${\mathcal F}$ of $k_2$-element subsets of $[\ell]$ has {\em Property~A$(k_1,k_2,\ell)$} (A is for asymmetric) if there is a $k_1$-element subset of $[\ell]$ that intersects every set in ${\mathcal F}$.
We then define $\overline{M}(k_1,k_2,\ell)$ to be the cardinality of a smallest family of $k_2$-element subsets of $[\ell]$ that does not have Property~A$(k_1,k_2,\ell)$. 
Note that this definition of $\overline{M}$ coincides with the definition given before the statement of Theorem~\ref{thm:completesteiner}.


\begin{proof}[Proof of Theorem~\ref{thm:completesteiner}]

	We define a $(k_A,k_B)$-list-assignment $L$ as follows. 
	Let ${\mathcal F}_1$ be a family of $\overline{M}(k_1,k_A,\ell)$ $k_A$-element subsets of $[\ell]$ without Property~A$(k_1,k_A,\ell)$.
	Let ${\mathcal F}_2$ be a family of $\overline{M}(k_2,k_B,\ell)$ $k_B$-element subsets of $[\ell]$ without Property~A$(k_2,k_B,\ell)$.
	 Assign the sets of ${\mathcal F}_1$ as lists to the vertices in $A$ and the sets of ${\mathcal F}_2$ as lists to the vertices in $B$. 
	 Suppose that $c$ is an $L$-colouring. Then $C_1=\{c(a) \mid a\in A\}$ intersects every set in ${\mathcal F_1}$ and so $|C_1|\ge k_1+1$ by assumption and similarly $C_2=\{c(b) \mid b\in B\}$ has cardinality $|C_2|\ge k_2+1$.
	 So $|C_1|+|C_2|\ge k_1+k_2+2 > \ell$, implying that $c$ cannot be a proper colouring.
\end{proof}

Since $\overline M(\ell-k_A,k_A,\ell)=\binom{\ell}{k_A}$, we have the following corollary of Theorem~\ref{thm:completesteiner}.
\begin{corollary}\label{cor:completesteiner}
Let $k_A,k_B,\ell$ be integers such that $k_A,k_B\ge 2$ and $\ell \ge k_A+k_B-1$.
The complete bipartite graph $G=(V{=}A\cup B,E)$ with $|A|=\binom{\ell}{k_A}$ and $|B|=\overline{M}(k_A-1,k_B,\ell)$ is not $(k_A,k_B)$-choosable.
\end{corollary}

\noindent
Although they are trivial, the following observations already show that Theorem~\ref{thm:completesteiner} and Corollary~\ref{cor:completesteiner} cannot be improved much in general. 
\begin{itemize}
\item
If $k_A=k_B=k$ and $\ell=2k-1$, then $\overline{M}(k_A-1,k_B,\ell)=\binom{2k-1}{k} =2^{2k+o(k)}$ and Corollary~\ref{cor:completesteiner} implies that $K_{\binom{2k-1}{k},\binom{2k-1}{k}}$ is not $k$-choosable. This implies $\ch(K_{n,n}) \gtrsim \frac12\log_2 n$, which one can compare to the bounds of Theorem~\ref{thm:ERT}.
\item
If $\ell=k_A\cdot k_B$, then $\overline{M}(k_A-1,k_B,\ell)=k_A$ and so Corollary~\ref{cor:completesteiner} implies that $K_{\binom{k_A\cdot k_B}{k_A},k_A}$ is not $(k_A,k_B)$-choosable.
This produces a necessary condition on $k_B$ for $(k_A,k_B)$-choosability only slightly weaker than Proposition~\ref{prop:classic}.
\end{itemize}

\noindent
We present several more instances where Theorem~\ref{thm:completesteiner} and Corollary~\ref{cor:completesteiner} are nearly sharp in Section~\ref{sec:asymptoticcomplete}.
For this, we will have use for the following estimates for the parameter $\overline{M}$. A version of this result can be found in~\cite[Ch.~13]{ES74}, but for completeness we present a standard derivation in the appendix.

\begin{theorem}\label{thm:PropA}
Let $k_1,k_2,\ell$ be integers such that $k_1,k_2\ge 2$ and $\ell\ge k_1+k_2$. Then
\begin{align*}
\frac{\ell!(\ell-k_1-k_2)!}{(\ell-k_2)!(\ell-k_1)!} \le \overline{M}(k_1,k_2,\ell) < \frac{\ell!(\ell-k_1-k_2)!}{(\ell-k_2)!(\ell-k_1)!}\log\binom{\ell}{k_1}.
\end{align*}
\end{theorem}

\section{Sufficient conditions in the complete case}\label{sec:SufComplete}

In this section, we give general sufficient conditions for $(k_A, k_B)$-choosability of complete bipartite graphs $K_{a,b}.$
Our strategy for establishing $(k_A,k_B)$-choosability in this setting is to take a random bipartition of the set of all colours and try to use one part for colouring $A$ and the other part for colouring $B$. This yields the following lemma.

\begin{lemma}\label{lem:completeupper}
	Let the reals $0 < \eps, p  < 1$ and positive integers $a$, $b$, $k_A$, $k_B$ satisfy either
	\begin{align}
	a p^{k_A}+b(1-p)^{k_B}<1 &\text{ or}\label{eqn:completeupper1}\\
	\frac{ap^{k_A-1}}{(1-\eps)k_B} + b \exp{(- \eps^2 k_B p /2)}<1 &.\label{eqn:completeupper2}
	\end{align}
	Then the complete bipartite graph $G=(V{=}A\cup B,E)$ with $|A|=a$ and $|B|=b$ is $(k_A,k_B)$-choosable.
\end{lemma}

\begin{proof}
	Let $L$ be $(k_A,k_B)$-list-assignment of $G$. Let $U = \cup_{v\in V}L(v)$, i.e.~$U$ is the union of all colour lists.
	Define a random partition of $U$ into parts $L_A$ and $L_B$ as follows. For each colour in $U$, randomly and independently assign it to $L_B$ with probability $p$ and otherwise assign it to $L_A$.
	
	First assume~\eqref{eqn:completeupper1}.
	For a given vertex $v\in A$, the probability that $L(v)\cap L_A=\emptyset$ is $p^{k_A}$.
	For a given vertex $v\in B$, the probability that $L(v)\cap L_B=\emptyset$ is $(1-p)^{k_B} $. So the expected number of vertices $v$ which cannot be coloured with the corresponding list $L_A$ or $L_B$ is equal to $a p^{k_A}+b(1-p)^{k_B}.$
	So the probabilistic method then guarantees some choice of the parts $L_A$ and $L_B$ such that every vertex in $A$ can be coloured with a colour from $L_A$ and every vertex in $B$ with a colour from $L_B.$
	
	Otherwise assume~\eqref{eqn:completeupper2}.
	For a given vertex $w\in B$, the random variable $|L(w)\cap L_B|$ has a binomial distribution of parameters $k_B$ and $p$. So by a Chernoff bound the probability that $|L(w)\cap L_B|$ is smaller than $(1- \eps) k_B p $ is smaller than $\exp{(- \eps^2 k_B p /2)}$.
	Thus, with probability greater than $1-b \exp{(- \eps^2 k_B p /2)}$, we have $|L(w)\cap L_B| \ge (1- \eps) k_B p$ for all $w\in B$.
	
	For a given vertex $v\in A$, the probability that $L(v)\cap L_A=\emptyset$ is $p^{k_A}$. So the expected number of vertices $v$ for which this holds is 
	$p^{k_A}|A|$.
	Thus by Markov's inequality, the probability that there are at least $ (1- \eps) k_B p$ such vertices $v$ is smaller than $\frac{p^{k_A-1}a}{(1-\eps)k_B}$.
	
	Due to~\eqref{eqn:completeupper2}, the probabilistic method guarantees a fixed choice of partition of $U$ into parts $L_A$ and $L_B$ such that $|L(w)\cap L_B| \ge (1- \eps) k_B p$ for every $w\in B$ and the number of vertices $v\in A$ such that $L(v)\cap L_A= \emptyset$ is smaller than $(1- \eps) k_B p$. Colour any vertex $v\in A$ having $L(v)\cap L_A= \emptyset$ with an arbitrary colour from $L(v)$. Colour any other vertex $v'\in A$ with a colour from $L(v')\cap L_A$. Finally colour any vertex $w\in B$ with some colour in $L(w) \cap L_B$ unused by any vertex of $A$ --- this is possible as there were fewer than $(1- \eps) k_B p$ colours of $L_B$ used to colour vertices in $A$ and the lists in $B$ all have at least $(1- \eps) k_B p$ colours from $L_B$.
	
	In either case, we are guaranteed a proper $L$-colouring of $G$, as promised.
\end{proof}

\section{Asymptotic sharpness in the complete case}\label{sec:asymptoticcomplete}

We next use the results of Sections~\ref{sec:completesteiner} and~\ref{sec:SufComplete} to roughly settle the behaviour of complete bipartite graphs with respect to Problem~\ref{prob:asymmetric} in several regimes. 
The conditions in Theorems~\ref{thr:3CasesChoos} and~\ref{thm:completecomparison} naturally correspond to the conditions in Conjecture~\ref{conj:general}.

\begin{theorem}\label{thr:3CasesChoos}
	Let the positive integers $a$, $b$, $k_A$, $k_B$ satisfy one of the following.
	\begin{enumerate}
		\item\label{itm:uncrossedcomplete,suff}
		For any $\eps>0$, $a,b\ge\Delta_0$ for some $\Delta_0=\Delta_0(\eps)$, $k_A > b^\eps$ and $k_B > a^\eps$.
		\item\label{itm:crossedcomplete,suff}
		For any $t>0$, $k_A \ge \frac{1}{t}\log_2 (2a)$ and $k_B > 2^t \log (2b)$.
		\item\label{itm:symmetricasymmetriccomplete,suff}
		We have that $a=b=\Delta$ and either $k_B > 8 (\Delta/(2\log(2\Delta)))^{1/k_A}\log(2\Delta)$ or $k_A > 8 (\Delta/(2\log(2\Delta)))^{1/k_B}\log(2\Delta)$.
	\end{enumerate}
	Then the complete bipartite graph $G=(V{=}A\cup B,E)$ with $|A|=a$ and $|B|=b$ is $(k_A,k_B)$-choosable.
\end{theorem}

\begin{proof}	
	Assume condition~\ref{itm:uncrossedcomplete,suff} and assume without loss of generality that $a \ge b$.
	Fix $\eps>0$ and take $\Delta_0>2$ large enough such that $\eps\Delta_0^\eps  > 3$ and $x^{\eps/3}> \log (2 x)$ for every $x \ge \Delta_0.$
	Now $k_B \ge a^{\eps}> a^{2/b^\eps} a^{\eps/3} > (2a)^{1/b^\eps} \log (2b)$. By taking $p= (2a)^{-1/b^\eps}$, we have that~\eqref{eqn:completeupper1} is satisfied and the result follows from Lemma~\ref{lem:completeupper}.
	
	Assume condition~\ref{itm:crossedcomplete,suff}. Take $p = 1/2^t$. Then $a p^{k_A} \le 1/2$ and $b (1-p)^{k_B} \le b \exp(-p k_B) < 1/2$, so that~\eqref{eqn:completeupper1} is satisfied and the result follows from Lemma~\ref{lem:completeupper}.

	Assume condition~\ref{itm:symmetricasymmetriccomplete,suff} and by symmetry assume $k_B > 8 (\Delta/(2\log(2\Delta)))^{1/k_A}\log(2\Delta)$.
	Take $p = 8\log (2\Delta)/k_B$ and $\eps=1/2$. 
	Then $p<(2 \log(2\Delta) / \Delta)^{1/k_A}$ and
	$b \exp{(- \eps^2 k_B p /2)} < 1/2$. Also
	\begin{align*}
	\frac{ap^{k_A}}{(1-\eps)pk_B} = \frac{\Delta p^{k_A}}{4\log(2\Delta)} < \frac12.
	\end{align*}
	Thus~\eqref{eqn:completeupper2} is satisfied and the result follows from Lemma~\ref{lem:completeupper}.
	%
	%
\end{proof}

The conditions are also sharp in some sense.

\begin{theorem}\label{thm:completecomparison}
For each of the following four conditions, there are infinitely many choices of the positive integers $a$, $b$, $k_A$, $k_B$ satisfying it such that
 the complete bipartite graph $G=(V{=}A\cup B,E)$ with $|A|=a$ and $|B|=b$ is not $(k_A,k_B)$-choosable.

\begin{enumerate}
\item\label{itm:uncrossedcomplete,nec}
Given a monotone real function $g$ satisfying $g(1) \ge 1$ and $g(x)=\omega(1)$ as $x\to\infty$, 
\begin{align*}
k_A > b^{1/g(b)} \text{ and } k_B=a.
\end{align*}
\item\label{itm:crossedcomplete,nec}
For any integer $t\ge 4$, either
\begin{align*}
k_A \le \frac{ \log_2 a }{t} \text{ and } k_B \le  \frac{2^{t-1}}{e} (\log{b} - \log{  \log{a}}) \le \frac{ \log_2 a }{t},
\end{align*}
or
\begin{align*}
k_A \le 2^t \log_2 a  \text{ and } k_B \le \frac{\log_2{b}- \log_2{  \log{a}}}{t+\log_2 t + 3},
\end{align*}
where $4 \le b \le a$. 
\item\label{itm:symmetricasymmetriccomplete,nec}
For any fixed integer $k>1$ and $\Delta$ sufficiently large, $\Delta_A=\Delta_B=\Delta$, $k_A=k$ and $k_B= c (\Delta/\log \Delta )^{1/k} \log \Delta $ for some constant $c=c(k)$.
\end{enumerate}

\end{theorem}

\begin{proof}
Consider condition~\ref{itm:uncrossedcomplete,nec}.
Take $\delta$ large enough so that $g(\delta^a)>a$.
Let $k_A=\delta$ and $b=\delta^a$ and note that $b^{1/g(b)}=\delta^{ a/g(b)} < \delta=k_A$, i.e.~$k_A^{g(b)}> b$. The result then follows from Proposition~\ref{prop:classic}. 

Consider condition~\ref{itm:crossedcomplete,nec}.
For the former case, let
\begin{align*}
\ell= \frac{2^t}{et}\log_2 a
\quad\text{and}\quad
k_A = \frac{ \log_2 a }{t} \ge k_B = \frac{2^{t-1}}{e} (\log{b} - \log{  \log{a}} ).
\end{align*}

Then $\binom{\ell}{k_A} \le ( e \ell/k_A )^{k_A} = a.$
Also 
\begin{align*}
\binom{\ell}{k_B} \left/ \binom{\ell-k_A+1}{k_B}\right. 
&\le \left( \frac{\ell - k_B}{\ell-k_A-k_B}\right)^{k_B}
\le \left( \frac{2^t-e}{2^t-2e} \right)^{k_B}
\le \exp \left( \frac{2e}{2^t} k_B \right)
= \frac{b}{\log a}.
\end{align*}
Here we used the estimation $(x-e)/(x-2e) \le \exp(2e/x)$ for $x \ge 16.$

For the latter case, choose
$$
\ell=\left( 2^t +\frac{1}{t+\log_2 t+3}\right) \log_2 a,
\quad
k_A = 2^t \log_2 a
\quad\text{and}\quad
k_B = \frac{\log_2{b}- \log_2{  \log{a}}}{t+\log_2 t + 3}.
$$
Then $\binom{\ell}{k_A}=\binom{\ell}{\ell-k_A} \le ( e \ell/(\ell-k_A) )^{\ell-k_A} \le a$ since
\begin{align*}
e \ell/(\ell-k_A) 
&\le e(2^t(t+\log_2 t+3 ) +1)
\le 8t 2^t = 2^{t+\log_2 t+3}.
\end{align*}
Also
\begin{align*}
\binom{\ell}{k_B}\left/\binom{\ell-k_A+1}{k_B}\right. 
&\le \left( \frac{e \ell}{\ell-k_A}\right)^{k_B}
\le \frac{b}{\log a}.
\end{align*}
In either case, the result follows from Corollary~\ref{cor:completesteiner} and Theorem~\ref{thm:PropA}.

Consider condition~\ref{itm:symmetricasymmetriccomplete,nec}. 
	First we make a computation verifying $\overline M ((\log m)/2, m/2, m) < m$ when $m$ is sufficiently large.
By Theorem~\ref{thm:PropA}, for $m$ sufficiently large, we have that 
\begin{align*}
\overline M ( (\log m)/2, m/2, m) 
&\le \left( \frac{m}{m/2-(\log m)/2} \right)^{(\log m)/2} \log \binom{m}{ (\log m)/2 }
< \sqrt{m} (\log m)^2
< m.
\end{align*}

We choose $m=c (\Delta/\log \Delta )^{1/k} \log \Delta$, where $c=1/(4k(k-1))$.
By the above computation, this choice satisfies $\overline M ( (\log m)/2, m/2, m) \le \Delta$ for $\Delta$ large enough.
Let $b$ be such that $(k-1)b \log \Delta = (\log m)/2$, and note that $b \sim 2c$ as $\Delta\to\infty$.

Let  $G=(A \cup B,E)$ be a complete bipartite graph with
\[
\lvert A \rvert = \overline M ( (\log m)/2 ,m/2, m)
\quad\text{and}\quad
\lvert B \rvert= b \log \Delta \binom{ \lceil  m/(b \log \Delta) \rceil }{k}.
\]
Note that as $\Delta\to\infty$
\[
\lvert B \rvert \le b \log \Delta\frac{\lceil  m/(b \log \Delta) \rceil^k}{k!} \sim \frac{1}{k!} \left(\frac{c}{b}\right)^k b \Delta \le \Delta
\]
and thus $\Delta_A,\Delta_B\le \Delta$ for all $\Delta$ large enough.

	We define a $(k_A,k_B)$-list-assignment $L$ as follows. 
	Let $\mathcal F$ be a family of $\lvert A \rvert$ $(m/2)$-element subsets of $[m]$ without Property~A$( (\log m)/2 ,m/2, m)$ and
	 assign the sets of $\mathcal F$ as lists to the vertices in $A$. So there is no $((\log m)/2)$-element subset of $[m]$ that intersects every list in $A$.
	 For $B$, arbitrarily partition $[m]$ into $b\log \Delta$ segments of nearly equal size, and assign as lists to the vertices in $B$ all possible $k$-element subsets chosen from within a single segment.
	 
	 Note that for any $L$-colouring, $\cup_{w\in B} L(w)$ intersects at least one colour from each $k$-element subset of a segment, and so $\cup_{w\in B} L(w)$ avoids at most $(k-1)b \log \Delta = (\log m)/2$ colours of $[m]$. However, as noted above, $\cup_{v\in A} L(v)$ must have more than $(\log m)/2$ colours of $[m]$, and this precludes a proper $L$-colouring.
%
\end{proof}

\section{Sharpness in a boundary complete case}\label{sec:boundary}

		In this section, we precisely solve Problem~\ref{prob:asymmetric} for complete bipartite graphs when $k_A \ge \Delta_A -1.$
		When $k_A \ge \Delta_A +1$, we know that any bipartite $G=(A \cup B,E)$ must be $(k_A,k_B)$-choosable.
		The case $k_A = \Delta_A$ is handled by Proposition~\ref{prop:classic} and the fact that its conclusion fails if $|A|<\delta^k$ or $|B|<k$. The remainder of this section is devoted to the case $k_A=\Delta_A-1$.

For this it will be useful to have the following simple lemma at our disposal.
Given a family $\mathcal{F}$ of disjoint subsets of $X$, a {\em transversal} of $\mathcal{F}$ is a subset of $X$ that intersects every set in $\mathcal{F}$ exactly once. An {\em almost-transversal} of $\mathcal{F}$ is a subset of $X$ that intersects all but one set in $\mathcal{F}$ exactly once.

\begin{lemma}\label{lem:counting_tuples}
Suppose $\mathcal{F}$ consists of $b$ disjoint sets  $F_1, F_2,\dots, F_b$ such that $\lvert F_1 \rvert \ge \lvert F_2 \rvert \ge \ldots \ge \lvert F_b \rvert$. Then every almost-transversal of $\mathcal{F}$ is subset of at most $\lvert F_1\rvert$ transversals of $\mathcal{F}$.
As a corollary, if $\mathcal{F}_\star$ is a family of almost-transversals of $\mathcal{F}$ such that every transversal of $\mathcal{F}$ contains some element of $\mathcal{F}_\star$, then $\mathcal{F}_\star$ must contain at least $\lvert F_2\rvert \cdots \lvert F_b\rvert$ elements.
\end{lemma}

\begin{proof}
Any almost-transversal of $\mathcal{F}$ can be extended to a transversal of $\mathcal{F}$ by adding exactly one element from $F_j$ for some $j\in[b]$, and so there are $\lvert F_j\rvert \le \lvert F_1\rvert$ choices for this. The corollary follows directly by a union bound argument.
\end{proof}

		\begin{proposition}\label{prop:kge4ChoosabilityofKk+1t}
			Let $b \ge 5$ and $\delta =q(b-1)+r$ with $0 \le r \le b-2$ and $q$ being integers such that $\delta \gg b.$
			Then the complete bipartite graph $G=(V{=}A\cup B,E)$ with $|A|=a$ and $|B|=b$ is $(b-1,\delta)$-choosable
			if and only if 
			$$ a < \delta^{b-1}-((b-2)q+r)^{b-1-r}((b-2)q+r-1)^r.$$
			Note that for $1\ll b\ll \delta$, this bound on $a$ is approximately $(1-1/e) \delta^{b-1}.$
		\end{proposition}
		
		\begin{proof}
			We write $B=\{v_1,v_2,\ldots, v_b\}.$
			Let $L$ be a $(b-1,\delta)$-list-assignment of $G$.
			If some colour is in the list of at least three different vertices of $B$, or two colours are both in the lists of two disjoint pairs of vertices of $B$, then we can certainly $L$-colour $G$. (The $L$-colouring uses the common colour(s) for the respective vertices and an arbitrary $L$-colouring for the remaining vertices of $B$, followed by a greedy $L$-colouring of the vertices in $A$.)
			Let $L_{12}=L(v_1) \cap L(v_2)$ and analogously define $L_{ij}$ for each $1 \le i <j \le b$ and let $\ell_{ij}=\lvert L_{ij}\rvert$.
			Consider the family of sets induced by the index pairs of the non-empty $L_{ij}$. This must form an intersecting family, for otherwise we have an $L$-colouring of $G$ similarly as above. We consider two cases depending on this family being trivial or not.

\begin{enumerate}
\item (non-trivial family)
		Without loss of generality, we may assume that only $L_{12}, L_{13}$ and $L_{23}$ are non-empty among all $L_{ij}$.
		If there are $b-1$ colours which do not occur as the list of a vertex of $A$, such that the vertices of $B$ can coloured with these $b-1$ colours, then $G$ is $L$-colourable.
		So if $G$ is not $L$-colourable, every such collection of $b-1$ colours has to occur as the list of a vertex in $A$.
		There are $(\delta-\ell_{12}-\ell_{13})(\delta-\ell_{12}-\ell_{23})(\delta-\ell_{13}-\ell_{23})\delta^{b-3}$ combinations of $b$ colours which appear only once among the lists $L(v_i), 1\le i \le b$. By Lemma~\ref{lem:counting_tuples}, every $(b-1)$-list $L(a)$ for some $a \in A$ can forbid at most $\delta$ of these and so at least $(\delta-\ell_{12}-\ell_{13})(\delta-\ell_{12}-\ell_{23})(\delta-\ell_{13}-\ell_{23})\delta^{b-4}$ such vertices are needed to forbid all these colourings.
		There are also $(\ell_{12}+\ell_{13}+\ell_{23})\delta^{b-2}-(\ell_{12}\ell_{13}+\ell_{12}\ell_{23}+\ell_{13}\ell_{23})\delta^{b-3}$ other possible colourings of $B$ which do use exactly $b-1$ different colours. So these possible colourings must also be forbidden by some list $L(a)$ if $G$ were not $L$-colourable.			
		
		This implies that $G$ is $L$-colourable if 
		\begin{align*}
		a 
		&<(\ell_{12}+\ell_{13}+\ell_{23})\delta^{b-2}-(\ell_{12}\ell_{13}+\ell_{12}\ell_{23}+\ell_{13}\ell_{23})\delta^{b-3}\\
		&\quad+(\delta-\ell_{12}-\ell_{13})(\delta-\ell_{12}-\ell_{23})(\delta-\ell_{13}-\ell_{23})\delta^{b-4}\\
		&=\delta^{b-1}-(\ell_{12}+\ell_{13}+\ell_{23})\delta^{b-2}+(\ell_{12}+\ell_{13}+\ell_{23})^2\delta^{b-3}\\
		&\quad-(\ell_{12}+\ell_{13})(\ell_{12}+\ell_{23})(\ell_{13}+\ell_{23})\delta^{b-4}.
		\end{align*}
		By the AM--GM Inequality, the minimum of the last expression subject to a fixed sum $\ell_{12}+\ell_{13}+\ell_{23}$ occurs when $\ell_{12}=\ell_{13}=\ell_{23}$, and so the 
		minimum is attained when $\ell_{12}=\ell_{23}=\ell_{13}=\delta/4,$ leading to a bound of $\frac{11}{16} \delta ^{b-1}.$

\item (trivial family)
Without loss of generality, we assume only some of the $L_{1j}$ are non-empty and $\ell_{1b} = \min\{ \ell_{1j} \mid 2 \le j \le b\}.$
There are $$\delta^{b-1} - \prod_{j=2}^{b} (\delta - \ell_{1j})$$ possible ways to choose one colour from every $L(v_i), 2 \le i \le b$, in such a way that $L(v_1)$ does contain at least one of these chosen colours.
So if $A$ would not contain these combinations among the lists $L(a), a \in A$, one would be able to colour $G$ with these $b-1$ colours.

By Lemma~\ref{lem:counting_tuples}, one needs also 
$$\left(\delta - \sum_{j=2}^{b} \ell_{1j} \right)\prod_{j=2}^{b-1} (\delta - \ell_{1j})$$
lists $L(a), a \in A$, to make it impossible to colour $G$ by first colouring the vertices $B$ solely with colours not appearing among the $L_{1j}.$
Hence the number of disjoint lists $L(a)$ to make a list-colouring of $G$ impossible needs to be at least
$$\delta^{b-1} - \left( \sum_{j=2}^{b-1} \ell_{1j} \right) \prod_{j=2}^{b-1} (\delta - \ell_{1j}).$$
This expression is minimised when $\delta-\sum_{j=2}^{b-1} \ell_{1j}$ and every $ \ell_{1j}$ for $2 \le j \le b-1$ are equal to a number of the form $\lfloor \frac{\delta}{b-1} \rceil,$ as by an integral version of the AM--GM Inequality the subtracted product attains its maximum when all factors (which sum to a fixed amount) differ by at most $1$.
When $ b\ll \delta$ and $b \ge 5$, we have that this is approximately $(1-(1- 1/(b-1) )^{b-1}) \delta^{b-1}$, which is smaller than the value $\frac {11}{16} \delta^{b-1}$ obtained in the previous case.
Equality can be attained, i.e.~$G$ is not $(b-1, \delta)$-choosable when $a = \delta^{b-1}-((b-2)q+r)^{b-1-r}((b-2)q+r-1)^r$ (or larger) as we can take the lists as being exactly those mentioned for minimising the expression. \qedhere
\end{enumerate}
		\end{proof}
	
		The same analysis also gives the result for $b \in \{3,4\}.$
		When $b=3$, the bound for $a$ is $\lfloor \frac 34 \delta ^2 \rceil.$
		When $b=4$, the same analysis as in Proposition~\ref{prop:kge4ChoosabilityofKk+1t} gives that the bound for $a$ occurs in the first case (non-trivial family), resulting in the following detailed proposition.
	
	 		\begin{proposition}\label{prop:k3Choosabilityof4t}
	 			Let $\delta \ge 2$.
	 			The complete bipartite graph $G=(V{=}A\cup B,E)$ with $|A|=a$ and $|B|=4$ is
	 			$(3, \delta)$-choosable if and only if $$a < 
	 			\begin{cases}
	 		\frac{11}{16} \delta^3 &\text{ if } \delta \equiv 0 \pmod 4,\\
	 		\frac{11}{16} \delta^3+ \frac{3}{16} \delta+\frac 1{8} &\text{ if } \delta \equiv 1 \pmod 4,\\
	 		\frac{11}{16} \delta^3+ \frac{1}{4} \delta&\text{ if } \delta \equiv 2 \pmod 4,\\
	 		\frac{11}{16} \delta^3+ \frac{3}{16} \delta-\frac 1{8} &\text{ if } \delta \equiv 3 \pmod 4.\\
	 		\end{cases}
	 			$$
	 		\end{proposition}

\section{Sharper than complete bipartite}\label{sec:improvement}

In this section, we prove that complete bipartite graphs are not exactly extremal for Problem~\ref{prob:asymmetric}. The complete bipartite graph $K_{\delta ^k-1, k}$ is $(k,\delta)$-choosable, but there are bipartite graphs with $\Delta_A=k$ and $ \Delta_B $ smaller than $\delta ^k-1$ which are not $(k,\delta)$-choosable.

\begin{proposition}\label{prop:improvement}
	For any $\delta, k \ge 2$, there is a bipartite graph $G = (V =A \cup B,E)$ with parts $A$ and $B$ having maximum degrees $k$ and $f(\delta, k)<\delta ^k$, respectively, that is not $(k,\delta)$-choosable.
	Moreover, $f( \delta, k) \le \sum_{i=1}^{\delta} i^{k-1}$.
\end{proposition}

\begin{proof}
	We recursively construct bipartite graphs $G_i=(A_i \cup B_i, E_i)$ with parts $A_i$ and $B_i$ having maximum degree $k$ and $\sum_{j=1}^i (\delta-j+1)^{k-1}$. We simultaneously define a $(k, \delta)$-list-assignment $L_i$ of $G_i$ such that there is some vertex $b_i\in B_i$ which can only be given one of $\delta-i$ colours out of its list in any proper $L_i$-colouring.
	
	Let $G_1$ be the complete bipartite graph $K_{ \delta^{k-1},k}$, and write $B_1=\{v_1, \ldots, v_k\}$.
	For the vertices of $B_1$, we assign $k$ disjoint lists of length $\delta$, specifically, $L(v_j)=\{(j-1)\delta+1,\dots,j\delta\}$ for $j\in [k]$.
	For the vertices of $A_1$, we assign as lists all possible $k$-tuples drawn from $\{1\}\times \prod_{j=2}^k L(v_j)$.
	Since $b_1 := v_1$ cannot be given the colour $1$ in any proper $L_1$-colouring, the conditions are satisfied for $i=1$.
	
	For the recursion, assume $i\ge 1$ and take the disjoint union of $k$ copies of $G_i$, relabelling their $(k, \delta)$-list-assignments so that their colour palettes are mutually disjoint.
	So the parts $A_{i+1},B_{i+1}$ of the bipartition so far include the disjoint unions of the $k$ respective parts.
	Let $v_1,\dots,v_k$ be the $k$ copies of $b_i$, and for each $j\in [k]$ write $L'(v_j)$ for the set of $\delta-i$ colours to which the colour of $v_j$ is restricted by assumption. By relabelling, we may assume $1 \in L'(v_1)$.
	We now add $(\delta-i)^{k-1}$ new vertices to $A_{i+1}$ that are adjacent to every $v_j$.
	For these new vertices, we assign as lists all possible $k$-tuples drawn from $\{1\}$ and $L'(v_j)$ for $2\le j\le k$.
	This completes the definition of $G_{i+1}$ and $L_{i+1}$.
	By induction, $b_{i+1} := v_1$ may only be given a colour from $L'(v_1)\setminus \{1\}$ in any proper $L_{i+1}$-colouring, and moreover $b_{i+1}$ is of maximum degree in $B_{i+1}$. So $(G_{i+1},L_{i+1},b_{i+1}) $ satisfies the required conditions. This completes the recursive step.
	
	The graph $G:=G_\delta$ with parts $A=A_\delta$ and $B=B_\delta$ is not $(k,\delta)$-choosable, since by construction we may not give any colour to $b_\delta$ in any proper $L_\delta$-colouring. Furthermore, the maximum degrees in $A$ and $B$ are respectively $k$ and $\sum_{i=1}^{\delta} (\delta-i+1)^{k-1}=\sum_{i=1}^{\delta} i^{k-1}$, as required.
\end{proof}

\section{Degrees and $(k_A,k_B)$-choosability}\label{sec:degrees}

In this section, we give a condition on the minimum degree for concluding that a bipartite graph is not $(k_A,k_B)$-choosable. This is a reduction to the behaviour for complete bipartite graphs.

		\begin{theorem}\label{thm:degrees}
			Suppose  the complete bipartite graph $G_0=(V{=}A_0\cup B_0,E)$ with $|A_0|=a$ and $|B_0|=b$ is not $(k_A,k_B)$-choosable.
			Then any bipartite graph $G=(V{=}A\cup B,E)$ with parts $A$ and $B$ such that $|A|\le |B|$ and $B$ has minimum degree $\delta_B > 4ab \log{ 4a} \log{k_A}$ is not $(k_A,k_B)$-choosable.
		\end{theorem}
		
		\begin{proof}
			Let $\F_a$ and $\F_b$ be the collections of lists of sizes $a$ and $b$, respectively, that can be assigned to $A_0$ and $B_0$, respectively, to certify non-$(k_A,k_B)$-choosability of $G_0$.
			Let $p= 1/(4b \log{k_A}).$
			Randomly choose $X \subset A$, each vertex included independently with probability $p$.
			Then $\Exp(|X|)=p|A|$ and so by Markov's inequality, 
			$$\Pr(|X|>2p|A|)<\frac 12.$$
			Define a list-assignment $L_X$ of $X$, by assigning to every vertex of $X$ uniformly and independently a list of $\F_a$.
			Call a vertex $v$ in $B$ good if every member of $\F_a$ occurs as a list on a neighbour (in $X$) of $v.$
			For any $F\in \F_a$, let us say that {\em $v$ is not good due to $F$} if $F$ does not as occur as a list on a neighbour of $v$. Note that
			$$\frac{\Pr(v\text{ is not good due to }F)}{a} \le \left( 1- \frac p{a} \right)^{ \delta_B } \le \exp \left( - \frac p{a} \delta_B \right) < \frac{1}{4a}$$ implying that $$\Pr(v \text{ is not good}) < \frac 14.$$
			So by Markov's inequality, $$\Pr\left( |\{v \mid v \text{ is not good}\}| > |B|/2 \right) \le \frac{ \Exp(|\{v \mid v \text{ is not good}\}|)}{|B|/2} < \frac12.$$
			
			By the probabilistic method, there is some $X \subset A$ and a list-assignment $L_X$ of $X$ such that $|X| \le 2p|A|$ and there are at least $|B|/2$ good vertices.
			Fix this choice and let $B^*$ be the set of good vertices.
			
			Fix an arbitrary $L_X$-colouring $c_X$ of $X$. There are at most $k_A^{|X|}$ possibilities for the colouring $c_X$.
			Define a list-assignment $L_{B^*}$ of $B^*$, by assigning to every vertex of $B^*$ uniformly and independently a list of $\F_b$.
			Since every $v \in B^*$ is good, all lists of $\F_a$ occur in the neighbourhood of $v$ and at least one choice of a list in $\F_b$ would imply that $v$ cannot be properly coloured with a colour of that list.
			Hence the probability that every $v \in B^*$ can be properly $L_{B^*}$-coloured in agreement with $c_X$ is at most 
			$$\left( 1- \frac{1}{b}\right)^{|B^*|} < \exp\left( - \frac{|B|}{2b}\right).$$
			The probability that some proper colouring of $G$ can be completed given any $L_X$-colouring $c_X$ is smaller than 
			$$k_A^{|X|}\exp\left( - \frac{|B|}{2b}\right) \le \exp\left(2p|A|\log{k_A} - \frac{|B|}{2b}\right) \le 1.$$
			Thus by the probabilistic method there exists a list-assignment $L_{B^*}$ of $B^*$ such that  no proper $L_{B^*}$-colouring can be found in agreement with any $L_X$-colouring.
		\end{proof}

\section{Conclusion}

We have begun the investigation of an asymmetric form of list colouring for bipartite graphs.
In one direction, we have found good general sufficient conditions through connections to independent transversals and to the coupon collector problem.
This has incidentally yielded a non-trivial advance towards a difficult conjecture of Krivelevich and the first author.
In another direction, we have established broad necessary conditions through an unexpected link between the bipartite choosability of complete bipartite graphs and a classic extremal set theoretic or design theoretic parameter.
This link has fed naturally into the formulation of three attractive conjectures along these lines.
Because of the rich connections this problem has to other important areas of combinatorial mathematics, we are hopeful that further study will lead to novel insights.
We remark that Conjecture~\ref{conj:general} comprises three asymptotic parameterisations of Problem~\ref{prob:asymmetric} that we found most natural and interesting, all derived essentially from Theorem~\ref{thm:completesteiner}. There could be several other nice choices.
Because the terrain is new, there are many interesting angles we have not yet had the opportunity to fully explore.
	
	One possibility, based on the connection to combinatorial design theory, comes to mind. We have that $\overline{M}(2,q^2,q^2+q+1)=q^2+q+1$ for every prime power $q$ due to the finite projective planes. With a small modification of the substitution of this fact into Corollary~\ref{cor:completesteiner}, we obtain that the complete bipartite graph $K_{\frac16q^3(q+1)(q^2+q+1),q^2+q+1}$ is not $(3,q^2)$-choosable.
	On the other hand, Lemma~\ref{lem:completeupper} shows this is not that far from optimal, and in particular~\eqref{eqn:completeupper2} shows that $K_{q^6/(80\log q)^2,q^2+q+1}$ is $(3,q^2)$-choosable. It would be interesting to narrow the gap.
	For the specific case $q=2$, a quick computer search checks that $K_{20,7}$ is not $(3,4)$-choosable, but finding the largest $r$ such that $K_{r,7}$ is $(3,4)$-choosable seems difficult.

	\paragraph{Acknowledgement.} The authors wish to thank the anonymous referees for their careful reading and helpful comments.

\bibliographystyle{abbrv}
\bibliography{bilist}

\appendix

\section{Extremal analysis of approximate Steiner systems}

\begin{proof}[Proof of Theorem~\ref{thm:PropA}]
First we prove the lower bound.
Fix a family $\mathcal{F}$ of $k_2$-element subsets of $[\ell]$ with cardinality less than the leftmost expression. Choose $C$ a $k_1$-element subset of $[\ell]$ uniformly at random.
For any fixed $F \in \mathcal{F}$, we have
\begin{align*}
\Pr(F \cap C = \emptyset) 
&
 = \binom{\ell-k_2}{k_1}\left/\binom{\ell}{k_1}\right.
 = \frac{(\ell-k_2)!(\ell-k_1)!}{\ell!(\ell-k_1-k_2)!}. 
\end{align*}
By a union bound and the choice of cardinality of $\mathcal{F}$,
\begin{align*}
\Pr(F\cap C = \emptyset\text{ for some }F\in{\mathcal F}) \le
\sum_{F\in\mathcal{F}}\Pr(F \cap C = \emptyset) < 1.
\end{align*}
So with positive probability there is a set $C$ certifying that $\mathcal{F}$ has Property~A$(k_1,k_2,\ell)$.

Next we prove the upper bound.
Fix $C$ a $k_1$-element subset of $[\ell]$. Let $F$ be a $k_2$-element subset of $[\ell]$ chosen uniformly at random. Then
\begin{align*}
\Pr(F \cap C = \emptyset) 
& = \binom{\ell-k_1}{k_2}\left/\binom{\ell}{k_2}\right. = \frac{(\ell-k_2)!(\ell-k_1)!}{\ell!(\ell-k_1-k_2)!}. 
\end{align*}
Let $\mathcal{F} = \{F_1,\dots,F_m\}$ be a family of $m$ $k_2$-element subsets of $[\ell]$ chosen uniformly at random. Based on the above calculation,
\begin{align*}
\Pr( F_i \cap C \ne \emptyset\text{ for all }i\in\{1,\dots,m\}) \le \left(1-\frac{(\ell-k_2)!(\ell-k_1)!}{\ell!(\ell-k_1-k_2)!}\right)^m.
\end{align*}
There are $\binom{\ell}{k_1}$ choices for $C$, so we have
\begin{align*}
\Pr(\mathcal{F}&\text{ contains a $C$ certifying Property~A$(k_1,k_2,\ell)$}) \\
& \le \binom{\ell}{k_1}\exp\left(-m\frac{(\ell-k_2)!(\ell-k_1)!}{\ell!(\ell-k_1-k_2)!}\right).
\end{align*}
This last expression is less than $1$ if
\begin{align*}
m > \frac{\ell!(\ell-k_1-k_2)!}{(\ell-k_2)!(\ell-k_1)!}\log\binom{\ell}{k_1},
\end{align*}
which establishes the upper bound.
\end{proof}


\end{document}